\documentclass[11pt, oneside]{article}
\usepackage{geometry}                		
\usepackage{graphicx}				
\usepackage{amssymb}

\usepackage[english]{babel}
\usepackage{amsthm}    
\usepackage{amsfonts}  
\usepackage{eufrak}    
\usepackage{mathrsfs}  
\usepackage{amsmath}   

\usepackage{amsmath,amsfonts,amssymb,amsthm,epsfig,epstopdf,titling,url,array}

\theoremstyle{plain}
\newtheorem{thm}{Theorem}[section]
\newtheorem{lem}[thm]{Lemma}
\newtheorem{prop}[thm]{Proposition}
\newtheorem*{cor}{Corollary}

\theoremstyle{definition}
\newtheorem{defn}{Definition}[section]

\theoremstyle{remark}
\newtheorem{rem}{Remark}

   \usepackage{hyperref}

\begin{document}

\title{On a solution of a fractional hyper-Bessel \\ [3pt]
       differential equation by means of a \\ [3pt]
       multi-index special function\vspace*{-3pt}}
\author{\normalsize Riccardo Droghei}
\date{}
 \maketitle
\footnote{Fract Calc Appl Anal 24, 1559–1570 (2021). https://doi.org/10.1515/fca-2021-0065}




 \begin{abstract}

In this paper we introduce a new multiple-parameters (multi-index) extension 
of the Wright function that arises 
from an eigenvalue problem 
for a case of hyper-Bessel operator involving Caputo fractional derivatives.
We show that by giving particular values to the parameters involved 
in this special 
 function, this leads 
 to some known special functions (as the classical Wright function, the $\alpha$-Mittag-Leffler function, the Tricomi function, etc.)
 that on their turn appear as cases of the so-called multi-index Mittag-Leffler functions. 
  As an application, we mention that this new generalization Wright function
is an isochronous solution of a nonlinear fractional partial differential equation.

 \medskip


 \smallskip


 \end{abstract}

 \vspace*{-16pt}

 \section{Introduction}\label{sec:1}

\setcounter{section}{1}

The special functions play an essential role in all 
fields of mathematical physics. In fact, in solving several problems,  
one is led to use various special functions, because the analytical solutions are expressed in terms of some of these functions. The applied scientists and engineers, dealing with practical application of differential equations,
recognize 
the role of the special functions as an important mathematical tool.
Recently, there has been an increasing interest to use special functions in mathematical models involving \textit{fractional order} differential equations and systems to investigate various physical, biological, biomedical, chemical, economical etc. phenomena. We can refer 
the reader to the books \cite{GM97, Kiry94, M10}, references therein, as well as surveys as \cite{Kiry10-SF} and \cite{Kiry21}.
These special functions, such as: Mittag-Leffler function, Wright functions,
their various extensions and more generally, the Fox $H$-functions,
 are called as Special Functions of Fractional Calculus (Sf of FC).

  \newpage 

Specially, the Wright function is one of the SF of FC which plays a prominent role in the solution of linear partial fractional differential equation. In particular, 
the Wright function appeared in articles related to partial differential equation of fractional order: boundary-value problems for the fractional diffusion-wave equation, e.g. \cite{GLM00}.
Initially, 
it was introduced and investigated in a series of number theory notes from 1933 and on (\cite{Wright1933,Wright1935a,Wright1935b}) in the framework of the theory of partitions.
Recently, 
several extensions 
of the Wright function have been proposed: the M-Wright function involved in 
a relevant class of self-similar stochastic processes that are 
generally referred 
 as time-fractional diffusion processes \cite{MMP10}; a four parameters extension of Wright function $W_{\alpha,\beta}^{\gamma,\delta}(z)$ introduced and studied in \cite{ES15}, etc.

  \smallskip

  In this paper we provide a solution of 
  a fractional differential equation involving hyper-Bessel type 
   operator. 
  We note that 
  the solution found, 
  in specific cases can be reduced to some 
  known SF of FC (as the Laguerre-exponential function, $n$-Mittag-Leffler function, classical Wright functions and Tricomi function, 
  hyper-Bessel functions, etc.) that on their turn are particular cases of multi-index Mittag-Leffler functions..

\smallskip

The paper is organized as follows. 
In Section \ref{sec:2}, 
we provide some preliminary information on the operators of fractional calculus and the hyper-Bessel type differential operators, and on some classes of the SF of FC, as multi-index Mittag-Leffler functions.
Then, in Section \ref{sec:3} 
introduce a fractional differential equation involving a fractional hyper-Bessel type 
operator 
We find a solution 
of this equation, as what we call as 
the \textit{multiple-parameters generalized Wright} functions $\mathcal{W}^{\left(\bar{\alpha},\bar{\nu}\right)}(z)$.
We prove that the powers series defining this special function  
represents an entire in the complex plane.
Then,  in 
Section \ref{sec:4}, 
we illustrate some cases for 
particular values 
of the parameters involved  in the function $\mathcal{W}^{\left(\bar{\alpha},\bar{\nu}\right)}(z)$,
thus retrieving 
some well-known SF of FC.
Finally, we mention about an application of the introduced special function as 
a isochronous solution of a partial differential equation involving Caputo fractional derivative.

\vspace*{-4pt}
\section{Preliminaries} 
\label{sec:2}

\setcounter{section}{2}

\vspace*{3pt} 

\subsection{Preliminaries on fractional calculus operators }\label{subsec:2.1}

In order to make the  paper self-contained, we briefly recall some basic definitions of fractional calculus operators.

Let $\gamma\in \mathbb{R}^{+}$. The Riemann-Liouville fractional integral is defined by
\vskip -8pt
\begin{equation} 
I^{\gamma}_x f(t) =
\frac{1}{\Gamma(\gamma)}\int\limits_0^{x}(x-x')^{\gamma-1}f(x') dx'. 
\label{riemann-l}
\end{equation}
 Note that, by definition, $I^0_x f(t)= f(t)$.  
Moreover it satisfies the semigroup property, i.e. $I_x^{\alpha} I_x^{\beta} f(t)= I_x^{\alpha+\beta}f(t)$. 

There are different definitions of fractional derivative (see e.g. \cite{Pod99}, \cite{KST06}).
In this paper we use the fractional derivatives in the sense of Caputo: 
\vskip -13pt 
\begin{equation}
D_x^{\gamma}f(x) \!=\!   I_x^{m-\gamma} D_x^m f(x) \!=\!
\frac{1}{\Gamma(m\!-\!\gamma)}\int\limits_0^{x}\! (x-x')^{m-\gamma-1}\frac{d^m}{d(x')^m}f
(x') \, \mathrm dx', \;\gamma \ne m.
\end{equation}


\subsection
{Differential operators of hyper-Bessel type and related special functions} \label{subsec:2.2}

In \cite{DR2003}, Dattoli, He and Ricci considered the so-called 
\textit{Laguerre derivative operator} 
\vskip -14pt 
\begin{equation}\label{HBR}
D_{nL}=\underbrace{\frac{d}{dx}x\dots\frac{d}{dx}x\frac{d}{dx}x\frac{d}{dx}}_{n+1\,\,\,derivatives}.
\end{equation}
Its fractional order version was 
analyzed  by Garra and Polito in \cite{GarraPolito}
as a \textit{hyper-Bessel type} operator of the form
\begin{equation}\label{HBGP}
\underbrace{\frac{d^\nu}{dx^\nu}x^\nu\dots\frac{d^\nu}{dx^\nu}x^\nu\frac{d^\nu}
{dx^\nu}x^\nu\frac{d^\nu}{dx^\nu}}_{n+1\,\,\,derivatives},
\end{equation}
\vskip -2pt \noindent 
with $x\geq0,\,\,\nu>0$ and $\dfrac{d^\nu}{dx^\nu}$ representing the Caputo fractional derivative, \cite{KST06, Pod99}.


It is to mention that the operators $D_{L} = {\dfrac{d}{dx} x \dfrac{d}{dx}}$ ($n=1$) and iterated $D_{nL}$ as in \eqref{HBR} have been treated with regard to operational calculi for them yet by Ditkin and Prudnikov, see for example in \cite{Ditkin-Prudnikov} (reflecting their studies since 1962 published in Russian, like \cite{Dit-Prud62}). And \eqref{HBR} is only a special case of the so-called \emph{hyper-Bessel differential operators} of higher integer order $m \geq 1$, introduced by Dimovski \cite{Dimovski1966} and studied by him in a series of papers since 1966:
\vskip -13pt
\begin{equation} \label{hB}
B = x^{\alpha_0} \frac{d}{dx} x^{\alpha_1} \frac{d}{dx} x^{\alpha_2} \dots \frac{d}{dx} x^{\alpha_{m-1}} \frac{d}{dx}
x^{\alpha_m}, \ \ x>0,
\end{equation}
with $\beta = m- (\alpha_0+\alpha_1 \dots + \alpha_m) > 0$. These appear as evident far generalizations of the Bessel differential operator $B_{\nu}$ of order $m=2$, when 
$\alpha_0=\nu^{'}\!-\!1,\, \alpha_1=-2\nu^{'}\!+\!1,\, \alpha_2=\nu^{'}$; 
\, $\nu^{'}:=\pm \nu$. Dimovski developed operational calculi for the corresponding linear right inverse operator $L$ of \eqref{hB}, the hyper-Bessel integral operator such that $BL=\mbox{Id}$, of both algebraic Mikusinski type as well as based on a suitable Laplace type integral transform, called Obrechkoff transform. For more details on the theory of the hyper-Bessel operators and its applications, see Kiryakova \cite[Ch.3]{Kiry94} and references therein.

\newpage 

In the book \cite[Ch.5]{Kiry94} and next papers, Kiryakova considered also \emph{fractional multi-order analogues of the hyper-Bessel operators} \eqref{hB}, say of the form (as (5.4.8) there),
\vskip -11pt
\begin{equation}\label{frhB}
\mathcal{D} = x^{\alpha_0} \left(\frac{d}{dx}\right)^{\delta_1} x^{\alpha_1} \left(\frac{d}{dx}\right)^{\delta_2} x^{\alpha_2} \dots \left(\frac{d}{dx}\right)^{\delta_{m-1}} x^{\alpha_{m-1}} \left(\frac{d}{dx}\right)^{\delta_m}
x^{\alpha_m},
\end{equation}
identified as cases of the generalized fractional derivatives of multi-order $\delta = (\delta_1,\dots, \delta_m)$.
Details can be found in \cite{Kiry94} and next, e.g. in \cite{KirFCAA2014}.

\subsection{Special functions of fractional calculus related to hyper-Bessel type operators} 

First of all, let us remind the \emph{classical Wright function}
\vskip -8pt
\begin{equation}\label{Wright}
W_{\beta, \nu} (x) = \sum\limits_{k=0}^{\infty} {\frac {x^k} {k!\,\Gamma (\beta k+\nu)}},
\end{equation}
and the 2 parameter \emph{Mittag-Leffler function}
\vskip -8pt
\begin{equation}\label{ML}
E_{\alpha, \beta} (x) = \sum\limits_{k=0}^{\infty} {\frac {x^k} {\Gamma (\alpha k+\beta)}}, \ \ \alpha >0.
\end{equation}

Then, the \emph{multi-index  Mittag-Leffler functions} were introduced and studied by Kiryakova
\cite{Kir2000}, \cite{Kiry10-mML}, \cite{Kiry21}, etc., by means of 2 sets of $m$ indices
$(\alpha_1 >0,..., \alpha_m >0)$ and $(\beta_1,...,\beta_m)$  instead of $\alpha, \beta$ in \eqref{ML}, as
\vskip -10pt
\begin{equation}\label{mML}
E_{(\alpha_i),(\beta_i)}^{(m)} (x) = \sum\limits_{k=0}^{\infty} {\frac {x^k} {\Gamma (\alpha_1 k+\beta_1) \dots  \Gamma (\alpha_m k+\beta_m)}},
\end{equation}
shown to incorporate a very large class of SF of FC, see above mentioned works.

Another kind of Mittag-Leffler type functions with more indices were introduced \emph{by Kilbas and Saigo} \cite{Kilbas-Saigo}, as solutions of a kind of Abel-Volterra integral equations:
\vskip -13pt
\begin{equation} \label{KilbasSaigo}
E_{\alpha,\mu,l} (x)=\sum\limits_{k=0}^{\infty} c_k x^k,  \ \ \ \mbox{with} \ \
c_k=\prod_{j=0}^{k-1}{\frac {\Gamma[\alpha(j\mu+l)+1]} {\Gamma(\alpha(j \mu+l+1)+1]}}.
\end{equation}

For the \emph{hyper-Bessel differential equations} with operator \eqref{hB} of arbitrary integer order $m \geq 1$,
\vskip -12pt
$$
B\, y(x) = \lambda\, y(x), \ \,  B\, y(x) = f(x), \ \, B\, y(x) = \lambda\, y(x) +f(x), \ \ \lambda \neq 0, \ \
f \not\equiv 0,
$$
the explicit solutions have been found by Kiryakova et al. \cite{Kiry94},\cite{Kir-McBride}, \cite{Kir-Saqabi}, etc.
in all these 3 cases. Specially, we need to remind the so-called \emph{hyper-Bessel functions of Delerue} \cite{Delerue}, 
$$ 
 J^{(m)}_{\nu_1,...,\nu_m} (x)
= {\frac {(x/m\!+\!1)^{\nu_1+...+\nu_m}} {\Gamma(\nu_1+1) ... \Gamma(\nu_m+1)}}\ j^{(m)}_{\nu_1,...,\nu_m} (x), \ \, \mbox{where}
$$ \vspace*{-7pt}
\begin{equation} \label{hBf}
 j^{(m)}_{\nu_1,...,\nu_m} (x)
= {}_0 F_m \left((\nu_k+1)_1^m; - (x/m+1)^{m+1} \right)
\end{equation}
\vspace*{-7pt}
$$
= \sum\limits_{k=0}^{\infty} {\frac {\Gamma(\nu_1+1)...\Gamma(\nu_m+1)} {k!\ \Gamma(\nu_1+k+1)...\Gamma(\nu_m+k+1)}}\,
\left( (- (x/m+1)^{m\!+\!1})^k \right).
$$
For $m=1$, we have naturally the \emph{Bessel function} 
\vskip -12pt 
\begin{equation}\label{Bessel}
J_{\nu} (x) = {\frac {(x/2)^{\nu}} {\Gamma(\nu+1)}}\, j_{\nu} (x) =
{\frac {(x/2)^{\nu}} {\Gamma(\nu+1)}} \sum_{k=0}^{\infty} {\frac {(-x^2/4)^k} {k!\, \Gamma(\nu+k+1)}}.
\end{equation}
 The related  ``normalized" hyper-Bessel functions $j^{(m)}_{\nu_1,...,\nu_m} (x)$ are shown to be eigenfunctions of the operator $B$ and the fundamental system of solutions of $By(x) = \lambda y(x)$  consists of $m$ functions of the form \eqref{hBf}. In particular (Kiryakova \cite[Ch.3]{Kiry94}), the IVP
 \vskip-11pt
 $$
 B\, y(x) = - y(x), \  \, y(0)=1, \, y^{'}(0)=y^{''}(0)=...=y^{(m-1)} (0)=0
 $$
in the case $\alpha_m=0$ has its unique solution $j^{(m-1)}_{\nu_1,...,\nu_{m-1}} (x)$.
 While, in works as \cite{Kir2000}, \cite{Ali-Kir} and next ones,  it is shown that the multi-index Mittag-Leffler function \eqref{mML} serves as a solution of a fractional hyper-Bessel differential equation of the form
 $\mathcal{D}\, y(x) = \lambda\, y(x)$ with a generalized derivative of fractional multi-order $(\delta_1,...,\delta_m)$ of the form \eqref{frhB}.

\section{Main results: multi-parameter generalized Wright function}
\label{sec:3} 

\setcounter{section}{3} 

In this paper we consider 
a  \textit{fractional hyper-Bessel operator} defined as
\begin{equation}\label{HLO}
D^{\left(\bar{\alpha},\bar{\nu}\right)}_{nL}=\frac{d^{\alpha_{n+1}}}{dx^{\alpha_{n+1}}}x^{\nu_n}\frac{d^{\alpha_{n}}}{dx^{\alpha_{n}}}x^{\nu_{n-1}}\frac{d^{\alpha_{n-1}}}{dx^{\alpha_{n-1}}}\cdots x^{\nu_1}\frac{d^{\alpha_{1}}}{dx^{\alpha_{1}}},
\end{equation}
 where $\bar\alpha=\left(\alpha_1,...,\alpha_{n+1}\right)$;  $\bar\nu=\left(\nu_1,...,\nu_{n}\right)$;  
 $\dfrac{d^{\alpha_{j}}}{dx^{\alpha_j}},\,\,j=1,...,n+1$ are Caputo fractional derivatives and $\alpha_j>0,\,\,j=1,..., n+1$ and $\nu_j>0,\,\,j=1,...,n$. %
 Note that this is alternative form (with our somewhat different notations!) of the operator \eqref{frhB}, and that
  (\ref{HBR}), (\ref{HBGP}), (\ref{hB}) are its particular cases.

\begin{defn}[\textbf{\textit{m-p generalized Wright function}}]
A multiple parameters generalized version of the Wright function $\mathcal{W}^{\left(\bar{\alpha},\bar{\nu}\right)}(z)$ is defined by the series representation as a function of the complex variable $z$ and parameters $\alpha_j,\,\,j=1,..., n+1$ and $\nu_j,\,\,j=1,...,n,$
\begin{eqnarray}\label{Wcomp} 
 &&\mathcal{W}^{\left(\bar{\alpha},\bar{\nu}\right)}(z)=\sum_{k=0}^\infty\prod_{i=1}^k\prod_{j=1}^n\frac{\Gamma(\alpha_{n+1} i+a_j)}{\Gamma(\alpha_{n+1} i+b_j)}\cdot\frac{z^k}{\Gamma(\alpha_{n+1} k+b_{n+1})},
\end{eqnarray}
\vskip -3pt \noindent 
where
\vskip -16pt
\begin{equation}
a_j=1+\sum_{m=1}^j\left(\nu_{m-1}-\alpha_{m}\right),\,\,\,\,\,\,\,\,
b_j=1+\sum_{m=1}^j\left(\nu_{m-1}-\alpha_{m-1}\right).\nonumber
\end{equation}

First 
we show that 
 $W^{\left(\bar{\alpha},\bar{\nu}\right)}(z)$ is an entire function for $\alpha_j>0,\,j=1,...,n+1;\,\,\nu_j\in \mathbb{C},\,j=1,...,n$ and $\alpha_0=\nu_0=0$.
\end{defn} 


\begin{lem} 
The multiple parameters generalized Wright function 
is an entire function of the complex variable $z$. 
\end{lem}

\begin{proof}
Using the Wendel asymptotic formula \cite{Wendel} (see also \cite{Erdely1} formula 1.18 (4)), we have
\vskip -14pt
\begin{equation}
\frac{\Gamma(z+a)}{\Gamma(z+b)}=z^{a-b}\left[1+O(z^{-1})\right],\  \,\, |z|\rightarrow\infty,\ \,\,|arg(z)|<\pi.
\end{equation}

One can 
 compute the convergence radius of the series for 
 the multiple parameters generalized Wright function using the D'Alembert Criteria (Ratio Test):
\vskip -14pt
$$ 
\lim_{k\rightarrow+\infty}|\frac{a_{k+1}}{a_k}|=\lim_{k\rightarrow+\infty}|
\frac{\Gamma(\alpha_{n+1}k+b_{n+1})}{\Gamma(\alpha_{n+1}(k+1)+b_{n+1})}
\prod_{j=1}^n\frac{\Gamma(\alpha_{n+1}(k+1)+a_j)}{\Gamma(\alpha_{n+1}(k+1)+b_j)}|
$$ 
\vspace*{-10pt} 
\begin{equation}
= \lim_{k\rightarrow+\infty}|(\alpha_{n+1}k)^{-\alpha_{n+1}}\prod_{j=1}^n(\alpha_{n+1}k)^{-\alpha_{j}}|
 = \lim_{k\rightarrow+\infty}|\prod_{j=1}^{n+1}(\alpha_{n+1}k)^{-\alpha_{j}}|=0,
\end{equation} 
\vskip -2pt \noindent 
assuming 
$\alpha_j>0,\, j=0,...,n+1$.
\end{proof}

\begin{thm}\label{mainth.} 
The multiple parameters generalized Wright function $\mathcal{W}^{\left(\bar{\alpha},\bar{\nu}\right)}(\lambda x^{\alpha_{n+1}})$ with $\lambda\in\mathbb{R},\,\,x\geq0,\,\,\alpha_j>0,\,\,j=1,...,n+1$ and $\nu_j>0,\,j=1,...,n$,
satisfies the following fractional differential equation involving fractional
hyper-Bessel type 
operator (\ref{HLO})
\vskip -8pt 
\begin{equation}\label{eigen}
D^{\left(\bar{\alpha},\bar{\nu}\right)}_{nL} W^{\left(\bar{\alpha},\bar{\nu}\right)}(\lambda x^{\alpha_{n+1}})=\lambda x^{\sum_{s=1}^{n} (\nu_s-\alpha_s)}W^{\left(\bar{\alpha},\bar{\nu}\right)}(\lambda x^{\alpha_{n+1}}).
\end{equation}
\end{thm}

\begin{proof}
Applying recursively $\left(n+1\right)$-times 
the involved Caputo derivatives on the function \eqref{Wcomp}, 
we find the following expression
\vskip -12pt 
\begin{eqnarray}
&&\frac{d^{\alpha_{n+1}}}{dx^{\alpha_{n+1}}}x^{\nu_n}\frac{d^{\alpha_{n}}}{dx^{\alpha_{n}}}x^{\nu_{n-1}}\frac{d^{\alpha_{n-1}}}{dx^{\alpha_{n-1}}}\cdots x^{\nu_1}\frac{d^{\alpha_{1}}}{dx^{\alpha_{1}}}W^{\left(\bar{\alpha},\bar{\nu}\right)}( \lambda x^{\alpha_{n+1}})\nonumber\\
&=&\frac{d^{\alpha_{n+1}}}{dx^{\alpha_{n+1}}} \sum_{k=1}^\infty
\prod_{i=1}^k  \prod_{j=1}^n\frac{\Gamma(\alpha_{n+1} i+a_j)}{\Gamma(\alpha_{n+1} i+b_j)\Gamma(\alpha_{n+1} k+b_{n+1})}
\nonumber 
\end{eqnarray} \begin{eqnarray} 
&& \times\,
\prod_{m=1}^n\frac{\Gamma(\alpha_{n+1} k+b_{m})}{\Gamma(\alpha_{n+1} k+a_{m})} \lambda^k x^{\alpha_{n+1} k-\alpha_1\cdots-\alpha_n+\nu_1\cdots+\nu_n} \nonumber\\
&=&\sum_{k=1}^\infty
 \prod_{i=1}^k
 \prod_{j=1}^n\frac{\Gamma(\alpha_{n+1} i+a_j)}{\Gamma(\alpha_{n+1} i+b_j)\Gamma(\alpha_{n+1} k+b_{n+1})}
 \nonumber \\ 
 && \times\,
 \prod_{m=1}^n\frac{\Gamma(\alpha_{n+1} k+b_{m})}{\Gamma(\alpha_{n+1} k+a_{m})}\frac{\Gamma(\alpha_{n+1} k+b_{n+1})}{\Gamma(\alpha_{n+1} k+a_{n+1})}\lambda^k x^{\alpha_{n+1} k+\sum_{s=1}^{n} (\nu_s-\alpha_s)}\nonumber\\
&=&\lambda x^{\sum_{s=1}^{n} (\nu_s-\alpha_s)}\sum_{k=1}^\infty\prod_{i=1}^{k-1}\prod_{j=1}^n\frac{\Gamma(\alpha_{n+1} i+a_j)}{\Gamma(\alpha_{n+1} i+b_j)}\frac{\lambda^{k-1}x^{\alpha_{n+1} (k-1)}}{\Gamma(\alpha_{n+1} (k-1)+b_{n+1})}\nonumber\\
&=&\lambda x^{\sum_{s=1}^{n} (\nu_s-\alpha_s)}\sum_{k=0}^\infty\prod_{i=1}^{k}\prod_{j=1}^n\frac{\Gamma(\alpha_{n+1} i+a_j)}{\Gamma(\alpha_{n+1} i+b_j)}\frac{\lambda^kx^{\alpha_{n+1} k}}{\Gamma(\alpha_{n+1} k+b_{n+1})}\nonumber\\
&=&\lambda x^{\sum_{s=1}^{n} (\nu_s-\alpha_s)}W^{\left(\bar{\alpha},\bar{\nu}\right)}( \lambda x^{\alpha_{n+1}}).
\nonumber
\end{eqnarray}
\end{proof} 

\vspace*{-16pt}

\begin{cor}\label{Cor3.1}
In the case if \ $\sum_{s=1}^{n} (\nu_s-\alpha_s)=0$, 
the multiple parameters generalized Wright function $\mathcal{W}^{\left(\bar{\alpha},\bar{\nu}\right)}(\lambda x^{\alpha_{n+1}})$ becomes eigenfunction of the fractional hyper-Bessel type 
operator (\ref{HLO}).
\end{cor}


\begin{rem}
As for further analysis, it would be interesting to compare then the solution from Corollary \ref{Cor3.1} with a multi-index Mittag-Leffler function of the kind of \eqref{mML}.
\end{rem}

\section{Some particular cases} \label{sec:4} 

\setcounter{section}{4} 

In this section, we assign particular values to the parameters of the function $\mathcal{W}^{\left(\bar{\alpha},\bar{\nu}\right)}(z)$  and so, 
can identify it with some special functions well known in literature.


\begin{rem}{\textbf{Laguerre-exponential function.}}\\ 
Let us 
consider the particular integer order case, when 
$\alpha_j=\nu_j=1,\,\,j=1,\cdots,n; \,\,\alpha_{n+1}=1$ and therefore, $a_j=0; b_j=1,\,\,j=1,\cdots,n+1$.

It is simply to verify that multiple parameters Wright function matches with the so-called Laguerre-exponential function
\vskip -13pt
\begin{equation}\label{Lexp}
e_{n}(x)=\sum_{k=0}^{\infty} \frac{x^k}{(k!)^{n+1}}
= \sum_{k=0}^{\infty} \frac{x^k}{\left(\Gamma(k+1)\right)^{n+1}}
= E^{(n+1)}_{(1,1,...,1), (1,1,...,1)} (x) , 
\end{equation}
\vskip -3pt \noindent
investigated in \cite{DR2003} as eigenfunction of operator \eqref{HBR},
which is practically a multi-index M-L function \eqref{mML} with $(n+1)$-parameters $(1,1,...,1)$ and $(1,1..,1)$.
\end{rem}


\begin{rem}{\textbf{$n$-Mittag-Leffler function.}}\\ 
Similarly, in case of $\alpha_i=\nu, \,i=1,...,n+1$ and $\nu_i=\nu,\, i=1,...,n$,\,
$n = 1, 2, 3,.... $, 
 the multiple parameters Wright function (\ref{Wcomp}) becomes the $n$-Mittag-Leffler function from \cite{GarraPolito}:
 \vskip -10pt 
\begin{equation}\label{nML}
E_{n;\nu,1}(x)=\sum_{k=0}^{\infty} \frac{x^k}{\left(\Gamma(\nu k+1)\right)^{n+1}},\,\,\,\,\,\,x\geq0,\,\,\,\nu>0.
\end{equation}
Moreover, the hyper-Bessel type 
 operator in (\ref{eigen}) is then 
 the hyper-Bessel-type operator \eqref{HBGP} 
 investigated in Garra and Polito \cite{GarraPolito}.
 Note that \eqref{nML} is again a case of the multi-index Mittag-Leffler function \eqref{mML}, since
 \vskip -10pt
 $$
 E_{n;\nu,1}(x)= \sum_{k=0}^{\infty} \frac{x^k}{\Gamma(\nu k+1)...\Gamma(\nu k+1)} = E^{(n+1)}_{(\nu,\nu,...,\nu),(1,1,...,1)} (x),
 $$
 and so, eigen function of \eqref{HBGP} as particular case of the fractional hyper-Bessel differential operators \eqref{frhB} and \eqref{HLO}.
\end{rem}


\begin{rem} 
Let us analyse 
{\bf the particular case of the m-p generalized Wright function} \eqref{Wcomp}  
\textbf{for $n=1$} with $\alpha_2=\beta,\,\, \alpha_1=\alpha$ and $\nu_1=\nu$:
\vskip -12pt
\begin{equation} \label{n=1} 
\mathcal{W}_{\alpha,\beta,\nu}(x^\beta)=\sum_{k=0}^\infty\prod_{i=1}^k\frac{\Gamma(\beta i+1-\alpha)}{\Gamma(\beta i+1)}\, \frac{x^{\beta k}}{\Gamma(\beta k+1-\alpha+\nu)}. 
\end{equation}  
\end{rem}

\begin{prop}\label{lag} 
Obviously, the particular case \eqref{n=1}, 
$f(x)\!=\!\mathcal{W}_{\alpha,\beta,\nu}(x^\beta)$, 
satisfies the following fractional differential equation
\vskip -9pt
\begin{equation}
\frac{d^\beta}{dx^\beta}\left(x^\nu\frac{d^\alpha}{dx^\alpha}f(x)\right)=x^{\nu-\alpha}f(x),
\end{equation}
involving two fractional derivatives in the sense of Caputo of orders $\alpha,\beta\in\left(0,1\right)$.
\end{prop}


\begin{rem}{\textbf{Classical Wright function.}}\\  
In the case $\alpha=1$ the function  $\mathcal{W}_{1,\beta,\nu}(\beta x^\beta)$ in \eqref{n=1}
 coincides with the classical Wright function \eqref{Wright}, namely: 
 \vskip -10pt
$$\mathcal{W}_{1,\beta,\nu}(\beta x^\beta)=\sum_{k=0}^{\infty}\frac{1}{k!\Gamma(\beta k+\nu)}x^{\beta k}=W_{\beta,\nu}(x^\beta), $$
which on its turn is a multi-index ($2 \!\times\! 2)$ M-L function as
\vskip -11pt
$$
W_{\beta,\nu}(x^\beta)= \sum_{k=0}^{\infty}\frac{x^{\beta k}}{\Gamma(1.k +\nu)\Gamma(\beta k+\nu)}
= E^{(2)}_{(1, \beta),(1,\nu)} (x^{\beta}).
$$
\end{rem}


\begin{rem}{\textbf{Tricomi function.}}\\ 
For $\alpha=\beta=\nu=1$ the function  $\mathcal{W}_{1,1,1}(x)$ matches with %
the \textit{Tricomi function}:
\vskip -15pt
 $$C_{0}(x)=\sum_{k=0}^{\infty}\frac{x^k}{\left(k!\right)^2} =
 \sum_{k=0}^{\infty} \frac{x^k}{\left(\Gamma(k+1)\right)^2}= E^{(2)}_{(1,1),(1,1)} (x).
 $$
According to Dattoli, He and Ricci \cite{DHR05}, this is 
an eigenfunction of the Laguerre derivative $D_L$, i.e. of \eqref{HBR} with $n=1$. 
It is also directly related to the Bessel differential operators with parameter $0$
and to the Bessel function \eqref{Bessel} of order $0$:\,
$ J_0(x)= C_0 \left(- x^2/4 \right)$.
\end{rem}


\begin{rem}{\textbf{Application to a nonlinear fractional isochronous PDE.}} 
As a simple application of Theorem \ref{mainth.} (and in particular Proposition \ref{lag}),
in \cite{DG20} we have analysed 
the following nonlinear fractional PDE with the remarkable property to have isochronous solutions, i.e. completely periodic solutions with fixed period $T=\frac{2\pi}{\omega}$, 
\begin{equation}\label{IsoLag}
\frac{\partial u(x,t)}{\partial t}+i\omega u(x,t)=\frac{\partial^\beta}{\partial x^\beta}x^{\nu}\frac{\partial^\alpha u(x,t)}{\partial x^\alpha} +ikx^{\nu-\alpha} u(x,t). 
\end{equation}
We have shown (\cite{DG20}) that \eqref{IsoLag} has 
an explicit separable variables isochronous solution as:
\vskip -13pt 
\begin{equation}
u(x,t)=\exp(i\omega t)\cdot \mathcal{W}_{\alpha,\beta,\gamma}(-ikx^\beta).
\end{equation}
\end{rem}

\section{Conclusion} 

 Our main result of this paper is related to a special 
 function obtained by solving the fractional differential equation (\ref{eigen})
 involving a fractional hyper-Bessel-type 
 operator (\ref{HLO}). 
  To the best of our knowledge, 
  such kind of a special function was not studied by now.
But as seen, it is reduced in particular cases to some  known special functions, which on their side are
cases of the Bessel and hyper-Bessel functions and  more generally, of the multi-index Mittag-Leffer functions of the form \eqref{mML}. 
The definition of the m-p generalized Wright function \eqref{Wcomp} introduced here also resembles somehow  to the Kilbas-Saigo function \eqref{KilbasSaigo} of M-L type.
Therefore, it is interesting to study the relations between the function $\mathcal{W}^{\left(\bar{\alpha},\bar{\nu}\right)}$ and the functions \eqref{hBf}, \eqref{mML}, \eqref{KilbasSaigo},
and in general -- with other SF of FC of Bessel and Mittag-Leffler type,
refer for example to Kiryakova \cite{Kiry21}. 

\vspace*{-5pt} 

\section*{Acknowledgements}

The author is grateful to Dr. Roberto Garra for providing essential information, to the editor who help me to enter the topic in deep, and helping to expand the bibliography.

\vspace*{-7pt}



 \bigskip \smallskip

 \it

 \noindent
Liceo Scientifico Francesco Severi\\
Viale Europa,36, 03100 Frosinone (FR), ITALY\\[4pt]
  e-mail: riccardo.droghei@posta.istruzione.it \\
  e-mail: riccardo.droghei@francescoseveri.org \\ [3pt] 

\end{document}